\documentclass[10pt, reqno]{amsart}
\RequirePackage{amssymb,amsmath,amsthm,color,graphicx} 

\RequirePackage[pdftex, colorlinks, colorlinks=true, citecolor=black, menucolor=black, linkcolor=black]{hyperref}
\hypersetup{
	colorlinks=true, 
	linkcolor=black, 
	urlcolor=blue, 
	linktoc=all 
}
\graphicspath{ {./Images/} }

\theoremstyle{plain}
\newtheorem{theorem}{Theorem}[section]

\newtheorem{definition}[theorem]{Definition}

\newtheorem*{theorem*}{Theorem}
\newtheorem*{proposition*}{Proposition}
\newtheorem*{definition*}{Definition}
\newtheorem*{lemma*}{Lemma}
\newtheorem*{corollary*}{Corollary}
\newtheorem*{example*}{Example}
\newtheorem*{remark*}{Remark}

\DeclareSymbolFont{yhlargesymbols}{OMX}{yhex}{m}{n} \DeclareMathAccent{\yhwidehat}{\mathord}{yhlargesymbols}{"62}
\numberwithin{equation}{section}

\newcommand{\C}{\mathbb{C}}
\newcommand{\R}{\mathbb{R}}
\newcommand{\Z}{\mathbb{Z}}

\renewcommand{\Im}{\operatorname{Im}}
\renewcommand{\Re}{\operatorname{Re}}

\renewcommand{\phi}{\varphi}

\newcommand{\qtq}[1]{\quad\text{#1}\quad}

\usepackage{enumitem}
\usepackage{amsrefs}
\usepackage{float} 
\usepackage{subcaption} 
\usepackage{lipsum} 
\usepackage{subcaption}

\begin{document}

\title[A finite-difference model of NLS soliton-potential interactions]{Modeling soliton-potential interactions for the 1-D cubic NLS via an invariant-preserving finite-difference method}

\author[A. Stone]{Anthony D. Stone Jr.}

\email{anthony\_stone35@mymail.eku.edu}

\author[C. Hogan]{Christopher C. Hogan}

\email{christopher.hogan@eku.edu}

\address{Department of Mathematics \& Statistics, Eastern Kentucky University, Richmond, KY, USA}

\begin{abstract}
    We consider the 1-D focusing cubic NLS with a smooth, spatially-dependent potential which decays rapidly as $x \rightarrow \pm \infty$. As is well known, under sufficient regularity conditions, solutions of the cubic NLS with such a spatially-dependent potential conserve both mass and energy. Additionally, in the absence of a potential, the 1-D focusing cubic NLS admits boosted soliton solutions. Furthermore, in the presence of a potential, the interaction of the potential with boosted soliton solutions to the free case is of active interest in the literature, particularly with respect to the dependence of this soliton-potential interaction on the velocity of the incident soliton. We adapt a finite-difference scheme to the setting of a smooth, rapidly-decaying spatially-dependent potential and show that this adaptation preserves the appropriate mass and energy invariants. Additionally, we use this algorithm to model soliton-potential interactions at a variety of incident soliton velocities.
\end{abstract}

\maketitle

\section{Introduction}\label{Intro}

In one spatial dimension, we consider the focusing cubic nonlinear Schr\"odinger equation (NLS) with an external potential: 

\begin{equation}\label{NLSV}
\begin{cases}
    i{\partial_t u} = - \frac{1}{2}{\partial^2_xu} + Vu - |u|^2u \\
    u(0,x) = u_0.
\end{cases}
\end{equation}
Under appropriate regularity conditions, the solution $u = u(t,x): \R \times \R \rightarrow \C$ satisfies the conservation laws

\begin{equation}\label{ContMassConservation}
M[u(t)] := \int_{\mathbb{R}} |u(t,x)|^2 \ dx \equiv M[u_0]
\end{equation}
and
\begin{equation}\label{ContEnergyConservation}
E[u(t)] := \int_{\mathbb{R}} \left(\frac{1}{4} |\nabla u(t,x)|^2 + \frac{1}{2} V(x)|u(t,x)|^2 - \frac{1}{4} |u(t,x)|^4 \right) \ dx \equiv E[u_0].
\end{equation}
Additionally, it is well-known that the underlying focusing cubic NLS,
\begin{equation}\label{NLS}
    i{\partial_t u} = - \frac{1}{2}{\partial^2_xu} - |u|^2u, \\
\end{equation}
with initial data of the form
\begin{equation}\label{NLSFreeSolitonData}
    u_0 = u(0,x) = e^{ixv}\operatorname{sech}(x - x_0) ,\quad x_0,v\in\mathbb{R},\footnote{Here, $v$ and $x_0$ are the velocity and translation parameters, respectively, of the soliton initial data \eqref{NLSFreeSolitonData} and corresponding free soliton solition \eqref{NLSFreeSolitonSol} to \eqref{NLS}.}
\end{equation}
produces soliton solutions
\begin{equation}\label{NLSFreeSolitonSol}
    u(t,x) = e^{ixv+{it}/{2}-itv^2/2} \operatorname{sech}(x-x_0-vt).
\end{equation}
In this work, we study an invariant-preserving finite-difference scheme originally developed by Delfour, Fourtin, and Payre in \cite{MR645840} and adapted to our setting, and we then use this adaptation to numerically model the evolution of soliton initial data under \eqref{NLSV} at a variety of velocities $v$.

\subsection*{Notation}\label{S: Notation}

Throughout, we let $\Delta x$ and $\Delta t$ represent our mesh sizes of space and time, respectively. Further, we let $m$ represent the spatial index and $n$ the time index, so that, for example, a restriction of a spacial function $f(x)$ to our discrete mesh nodes is written $f(m\Delta x)$. To keep notation compact when considering functions of space-time, we denote $u(n\Delta t,m\Delta x)$ via $u_m^n$. Further, when considering such a function with an emphasis on a particular index, we suppress the unnecessary index. Hence, the vector corresponding to the $n$-th instance in time is denoted via $u^n$. Finally, we notate matrices by the following convention: for $J, K \in \Z^{+}$, we denote the $j,k$-th element of the matrix $B \in \C^{J\times K}$ via $b_{j,k}$, where $j \in \{1, \dots, J\}$ and $k \in \{1, \dots, K\}$; hence $B = [b_{j,k}]$.

\subsection*{Background and motivation}\label{S: Background and motivation}
We were motivated to pursue this study by \cite{MR4714537}, which examined the dynamics of a highly boosted soliton (i.e. $v >> 1$) under \eqref{NLSV} with a potential which meets the following admissibility conditions (where $L^p$ denotes the usual Lebesgue space): 
\begin{definition}[Admissible, from \cite{MR4714537}]\label{admissible}
We say $V : \mathbb{R} \to \mathbb{R}$ is \emph{admissible} if the operator
\[
H := -\frac{1}{2}\partial_x^2 + V
\]
obeys the following:
\begin{itemize}
    \item $H$ has no zero energy resonance;
    
    \item $H$ has at most one simple negative eigenvalue $-\lambda$, with corresponding eigenfunction $\phi$ satisfying
    \[
    \phi \in L^1 \cap L^\infty;
    \]
    
    \item the potential obeys the pointwise bound
    \[
    |V(x)| \leq C \bigl(\sqrt{1 + x^2}\bigr)^{-s}
    \]
    for some $s > 2$.
\end{itemize}

We call $s$ the \emph{decay parameter} of $V$.
\end{definition}
\noindent In \cite{MR4714537}, the authors showed a fast soliton of the form \eqref{NLSFreeSolitonSol} will effectively be transmitted through the an admissible potential (as described in Definition~\ref{admissible}) under \eqref{NLSV}. 

A related low-velocity case was studied in the recent result \cite{caliaro2026confinementorbitalstabilitysolitons}, which examined slow soliton evolution on a family of graphs. Rather than imposing the conditions of Definition~\ref{admissible} on the potential, this result treats slow soliton-potential interactions on the line as a special case of a more general graph-theoretic result and showed that slow solitons are essentially reflected away from the potential.

Elsewhere in the literature, there is strong additional precedence for expecting the character of soliton-potential interactions to exhibit velocity-dependence. Concisely, it is well-known that, given appropriate initial data, the scattering theory of \eqref{NLSV} (see, eg. \cite{naumkinSharpAsymptoticBehavior2016}) centers around transmission and reflection coefficients 
which depend on the potential barrier $V$ and the soliton velocity $v$. These coefficients govern the ``splitting'' of appropriate data into pieces which are transmitted through the potential and reflected back by the potential, respectively. This phenomenon is particularly highlighted in the works which study various sub-cases of the $\delta$-potential case (where $v(x) = q\delta_0(x)$ for some $q \neq 0$) (for example \cites{datchevFastSolitonScattering2009, MR2335125, MR2318852, MR2342704}); in this $\delta$-potential setting, the transmission and reflection coefficients can be computed explicitly in terms of $v$ and the parameter $q$ (which controls the strength and character of the $\delta$-potential).

To our knowledge, for potentials satisfying the conditions of Definition~\ref{admissible}, the rigorous analysis of soliton-potential interactions at intermediate velocities is an open problem, even over finite time intervals such as those considered in \cite{MR4714537}. We thus view the outputs of our model at such velocity scales as merely indicative of expected behavior; rigorous theoretical treatment of the intermediate-velocity regime is necessary.

In this work, we prioritized consideration of a smooth, barrier-like nature of the potential which is strongly concentrated around the origin. Thus, for our computations, we adopted a simple geometric barrier potential found in \cite{PhysRevLett.129.040401} of the form
\begin{equation}\label{potential}
    V = \frac{h^2}{2mw^2} \biggl( \frac{1}{(1+(\frac{x}{w})^2)^2}\biggr).
\end{equation}
We will choose parameters $h = 0.3, m = 0.5,$ and $w = 0.1$. We selected this form for our potential because it is smooth and clearly decays as $x \rightarrow \pm \infty$, thereby aligning with conditions (2) and (3) of Definition~\ref{admissible}. This choice of potential has the added benefit of explicitly connecting the mathematical results noted above, particularly \cite{MR4714537} and \cite{caliaro2026confinementorbitalstabilitysolitons}, to potentials considered in the physics literature (eg. \cite{PhysRevLett.129.040401}). 

\subsection*{Results and organization thereof}\label{S: Results and Organization}
As the foundation for our numerical model, we adapt the finite-difference scheme introduced by Delfour, Fourtin, and Payre in \cite{MR645840}; thus, \eqref{NLSV} becomes:

\begin{equation}\label{Descrete_NLSV}
\begin{aligned}
i\frac{u_m^{n+1}-u_m^n}{\Delta t} =& - \frac{1}{4} \biggl[\frac{u_{m+1}^{n+1} - 2u_m^{n+1} + u_{m-1}^{n+1}}{|\Delta x|^2}\, + \frac{u_{m+1}^n - 2u_m^n + u_{m-1}^n}{|\Delta x|^2}\biggr] \\
&+ \biggl[{V(m\Delta x)} - \frac{|u_m^{n+1}|^2 + |u_m^n|^2}{2}\biggr]\frac{(u_m^{n+1} + u_m^n)}{2}.
\end{aligned}
\end{equation}

\noindent The approach of \cite{MR645840} provides two distinct advantages. The first advantage is purely pragmatic: the approach is simple to implement. The second advantage is more technical: both the original scheme of \cite{MR645840} and the adaptation of this algorithm which we describe below can be implemented to intrinsically preserve the appropriate mass and energy conservation properties, thus providing a straightforward verification mechanism for the accuracy of the results.

To this end, we state our main theoretical result, which says that solutions to \eqref{Descrete_NLSV} satisfy the the discrete analogs of \eqref{ContMassConservation} and \eqref{ContEnergyConservation}.

\begin{theorem}[Preservation of discrete conservation laws]\label{Discrete ConsLaws}
Let $\{u^n\}_{n = 1}^{\infty}$ satisfy \eqref{Descrete_NLSV}. Then for all $n \ge 0$,
\begin{equation}\label{DiscMassConservation}
M[u_m^n] = \sum_{m \in \Z} |u_m^n|^2 \Delta x = \sum_{m \in \Z} |u_m^0|^2 \Delta x = M[u_m^0] 
\end{equation}

and

\begin{equation}\label{DiscEnergyConservation}
\begin{aligned}
E[u_m^n] =\,& 
\frac{1}{4} \sum_{m \in \Z} \biggl|\frac{u_{m+1}^n - u_m^n}{\Delta x}\biggr|^2 \Delta x 
\\&+ \frac{1}{2} \sum_{m \in \Z}V(m\Delta x)|u_m^n|^2 \Delta x - \frac{1}{4} \sum_{m \in \Z} |u_m^n|^4 \Delta x
= E[u_m^0].
\end{aligned}
\end{equation}
\end{theorem}

The remainder of the paper is organized as follows. In section~\ref{S: Mass and Energy Conservation}, we prove Theorem~\ref{Discrete ConsLaws}. In section~\ref{S: computingsoliton-potential}, we describe the implementation of our adaptation of the scheme, highlighting the importance of the mass and energy conservation as a stopping condition. Finally, in section~\ref{S: Examples}, we present and discuss numerical results for soliton-potential interactions at three velocities which exhibit the expected in the three velocity regimes. These velocities are selected to highlight the effects of the incident velocity of the soliton data on the nature of soliton-potential interactions.

\section{Mass and energy conservation}\label{S: Mass and Energy Conservation}

In this section, we present the proof of Theorem~\ref{Discrete ConsLaws}. The proof technique here closely follows the proof presented in \cite{MR645840}, with modification to signs and the replacement of the potential term $\alpha \Delta x$ with a more general spatially-dependent potential $V(m\Delta x)$.
\begin{proof}[Proof of Theorem~\ref{Discrete ConsLaws}]
To show that the scheme satisfies the discrete conservation of mass, we multiply \eqref{Descrete_NLSV} by $(\bar{u}_m^{n+1} + \bar{u}_m^n)$, sum over the spatial index $m \in \Z$, and take the imaginary part of the result. We shall work term by term. Multiplying the left-hand term of \eqref{Descrete_NLSV} by $(\bar{u}_m^{n+1} + \bar{u}_m^n)$ yields
\begin{equation}\label{LHS_Discrete_NLSV}
T_m^n = i\frac{u_m^{n+1}-u_m^n}{\Delta t}(\bar{u}_m^{n+1} + \bar{u}_m^n) = \frac{i}{\Delta t}\bigl[\,|u_m^{n+1}|^2  - |u_m^n|^2 + 2i\Im(\bar{u}_m^{n+1}u_m^n) \bigr].
\end{equation}
Furthermore, multiplying the first term on the RHS of \eqref{Descrete_NLSV} by $(\bar{u}_m^{n+1} + \bar{u}_m^n)$ and regrouping yields
\begin{equation}\label{RHS_Discrete_NLSV_term_1}
\begin{aligned}
D_m^n =\,& - \frac{1}{4} \biggl[\frac{u_{m+1}^{n+1} - 2u_m^{n+1} + u_{m-1}^{n+1}}{|\Delta x|^2}\, + \frac{u_{m+1}^n - 2u_m^n + u_{m-1}^n}{|\Delta x|^2}\biggr](\bar{u}_m^{n+1} + \bar{u}_m^n) \\
=\,&  \frac{-1}{4|\Delta x|^2}\biggl\{\bigl[- 2|u_m^{n+1}|^2 - 2|u_m^n|^2 - 4\Re{(u_m^{n+1}\bar{u}_m^n)}\bigr] \\
&\quad\quad\quad\quad+ (u_{m+1}^{n+1}\bar{u}_m^{n+1} + u_{m-1}^{n+1}\bar{u}_m^{n+1}) + (u_{m+1}^n\bar{u}_m^n + u_{m-1}^n\bar{u}_m^n) \\
&\quad\quad\quad\quad+ (u_{m+1}^{n+1}\bar{u}_m^n  + u_{m-1}^n\bar{u}_m^{n+1}) + (u_{m+1}^n\bar{u}_m^{n+1} + u_{m-1}^{n+1}\bar{u}_m^n) \biggr\}.
\end{aligned}
\end{equation}
Finally, multiplying the second term on the RHS of \eqref{Descrete_NLSV} by $(\bar{u}_m^{n+1} + \bar{u}_m^n)$ and regrouping yields
\begin{equation}\label{RHS_Discrete_NLSV_term_2}
\begin{aligned}
V_m^n &= \biggl[{V(m\Delta x)} - \frac{|u_m^{n+1}|^2 + |u_m^n|^2}{2}\biggr]\frac{(u_m^{n+1} + u_m^n)}{2} (\bar{u}_m^{n+1} + \bar{u}_m^n) \\
&= \biggl[\frac{V(m\Delta x)}{2} - \frac{|u_m^{n+1}|^2 + |u_m^n|^2}{4}\biggr]\bigl(|\bar{u}_m^{n+1}|^2 + |\bar{u}_m^n|^2 + 2\Re{(u_m^{n+1}\bar{u}_m^n)}\bigr).
\end{aligned}
\end{equation}
We note that since by \eqref{RHS_Discrete_NLSV_term_2}, $V_m^n \in \R$ (and hence has no imaginary part, independent of our choice of $m \in \Z$), if we sum $V_m^n$ over all $m \in \Z$ and take the imaginary part of this sum, the result still vanishes:
\begin{equation}\label{Sum_RHS_Discrete_NLSV_term_2}
\Im{\biggl\{\sum_{m \in \Z} V_m^n\biggr\}} = 0.
\end{equation}
Next, we treat $D_m^n$ from \eqref{RHS_Discrete_NLSV_term_1}. Summing the result of \eqref{RHS_Discrete_NLSV_term_1} over $m \in \Z$, we have
\begin{equation}\label{Sum_Computation_RHS_Discrete_NLSV_term_1}
\begin{aligned}
\sum_{m \in \Z} D_m^n =\,&  \frac{-1}{4|\Delta x|^2}\biggl\{\sum_{m \in \Z} \bigl[- 2|u_m^{n+1}|^2 - 2|u_m^n|^2 - 4\Re{(u_m^{n+1}\bar{u}_m^n)}\bigr] \\
&+ \sum_{m \in \Z}(u_{m+1}^{n+1}\bar{u}_m^{n+1} + u_{m-1}^{n+1}\bar{u}_m^{n+1}) + \sum_{m \in \Z}(u_{m+1}^n\bar{u}_m^n + u_{m-1}^n\bar{u}_m^n) \\
&+ \sum_{m \in \Z}(u_{m+1}^{n+1}\bar{u}_m^n  + u_{m-1}^n\bar{u}_m^{n+1}) + \sum_{m \in \Z}(u_{m+1}^n\bar{u}_m^{n+1} + u_{m-1}^{n+1}\bar{u}_m^n) \biggr\}.
\end{aligned}
\end{equation}
Note that first sum on the RHS of \eqref{Sum_Computation_RHS_Discrete_NLSV_term_1} is real valued. Further,  by observing the conjugacy properties of subsequent terms in each of the remaining four sums, we see these sums are also real-valued. Hence, $\sum_{m \in \Z} D_m^n \in \R$, and we obtain 
\begin{equation}\label{Sum_RHS_Discrete_NLSV_term_1}
\Im{\biggl\{\sum_{m \in \Z} D_m^n\biggr\}} = 0.
\end{equation}
Finally, we address $T_m^n$. Note by a careful inspection of \eqref{LHS_Discrete_NLSV} we have that
\begin{equation}\label{Sum_Computation_LHS_Discrete_NLSV}
\Im{\biggl\{\sum_{m \in \Z} T_m^n \biggr\}} 
= \sum_{m \in \Z} \frac{1}{\Delta t} \big[|u_m^{n+1}|^2 - |u_m^{n}|^2\big].
\end{equation}
By \eqref{Descrete_NLSV}, we have that $T_m^n = D_m^n + V_m^n$, and similarly \begin{equation}\label{Im_Sum_Discrete_NLSV}
\Im{\biggl\{\sum_{m \in \Z} T_m^n\biggr\}} = \Im{\biggl\{\sum_{m \in \Z} D_m^n\biggr\}} + \Im{\biggl\{\sum_{m \in \Z} V_m^n\biggr\}}.   
\end{equation}
Hence, combining our results from \eqref{Sum_Computation_LHS_Discrete_NLSV}, \eqref{Sum_RHS_Discrete_NLSV_term_1}, and \eqref{Sum_RHS_Discrete_NLSV_term_2} via \eqref{Im_Sum_Discrete_NLSV}, we deduce that for all $n \geq 0$,
\begin{equation*}
\sum_{m \in \Z} \frac{1}{\Delta t} \big[|u_m^{n+1}|^2 - |u_m^{n}|^2\big] = 0,
\end{equation*}
which implies that for all $n \geq 0$, 
\begin{equation}\label{InductionHypothesis}
\sum_{m \in \Z} |u_m^{n+1}|^2 \Delta x = \sum_{m \in \Z} |u_m^{n}|^2 \Delta x.
\end{equation}
Therefore, we obtain \eqref{DiscMassConservation} by using \eqref{InductionHypothesis} together with a simple induction on $n$.

To show that the scheme satisfies discrete conservation of energy, we multiply \eqref{Descrete_NLSV} by $(\bar{u}_m^{n+1} - \bar{u}_m^n)$, sum over the spatial index $m \in \Z$, and take the real part of the result. Once again, we shall work term by term. Multiplying the left-hand term of \eqref{Descrete_NLSV} by $(\bar{u}_m^{n+1} + \bar{u}_m^n)$ yields
\begin{equation*}
\Tilde{T}_m^n = i\frac{u_m^{n+1}-u_m^n}{\Delta t}(\bar{u}_m^{n+1} - \bar{u}_m^n) = \frac{i}{\Delta t}\bigl[\,|u_m^{n+1}|^2  + |u_m^n|^2 - 2\Re(\bar{u}_m^{n+1}u_m^n) \bigr],
\end{equation*}
so that $\Tilde{T}_m^n\in i\R$ and hence has no real part, independent of our choice of $m \in \Z$. Summing $\Tilde{T}_m^n$ over all $m \in \Z$ and taking the real part of this sum, we have that
\begin{equation}\label{Sum_LHS_Discrete_NLSV_term_EC}
\Re{\biggl\{\sum_{m \in \Z} \Tilde{T}_m^n\biggr\}} = 0.
\end{equation}
Now, multiplying the first term on the RHS of \eqref{Descrete_NLSV} by $(\bar{u}_m^{n+1} - \bar{u}_m^n)$, then summing over all $m \in \Z$, regrouping, re-indexing, and regrouping again, we have
\begin{equation}\label{RHS_Discrete_NLSV_term_1_EC}
\begin{aligned}
\sum_{m\in\Z^{+}}\Tilde{D}_m^n =\,& - \frac{1}{4} \biggl[\frac{u_{m+1}^{n+1} - 2u_m^{n+1} + u_{m-1}^{n+1}}{|\Delta x|^2}\, + \frac{u_{m+1}^n - 2u_m^n + u_{m-1}^n}{|\Delta x|^2}\biggr](\bar{u}_m^{n+1} - \bar{u}_m^n) \\
=\,&  \frac{-1}{4|\Delta x|^2}\biggl\{\sum_{m\in\Z^{+}} \bigl[u_{m+1}^{n+1} - u_m^{n+1} + u_{m+1}^{n} - u_m^n\bigr] (\bar{u}_m^{n+1} - \bar{u}_m^n)\\
&\quad\quad\quad\quad- \sum_{m\in\Z^{+}} \bigl[u_m^{n+1} - u_{m-1}^{n+1} + u_m^n - u_{m-1}^n\bigr](\bar{u}_m^{n+1} - \bar{u}_m^n)\biggr\} \\
=\,&  \frac{-1}{4|\Delta x|^2}\biggl\{\sum_{m\in\Z^{+}} \bigl[u_{m+1}^{n+1} - u_m^{n+1} + u_{m+1}^{n} - u_m^n\bigr] (\bar{u}_m^{n+1} - \bar{u}_m^n)\\
&\quad\quad\quad\quad- \sum_{m\in\Z^{+}} \bigl[u_{m+1}^{n+1} - u_{m}^{n+1} + u_{m+1}^n - u_{m}^n\bigr](\bar{u}_{m+1}^{n+1} - \bar{u}_{m+1}^n)\biggr\} \\
=\,&  \frac{1}{4|\Delta x|^2}\biggl\{\sum_{m\in\Z^{+}} \bigl[(\bar{u}_{m+1}^{n+1} - \bar{u}_m^{n+1}) - (\bar{u}_{m+1}^n - \bar{u}_m^n)\bigr] \\
&\quad\quad\quad\quad\times \bigl[(u_{m+1}^{n+1} - u_m^{n+1}) + (u_{m+1}^{n} - u_m^n)\bigr] \biggr\} \\
=\,&\frac{1}{4} \sum_{m \in \Z} \biggl\{ \biggl|\frac{u_{m+1}^{n+1} - u_m^{n+1}}{\Delta x}\biggr|^2 - \biggl|\frac{u_{m+1}^n - u_m^n}{\Delta x}\biggr|^2 \\
&\quad\quad\quad\quad+ \frac{i}{2|\Delta x|^2} \Im{\bigl((\bar{u}_{m+1}^n - \bar{u}_m^n)(u_{m+1}^{n+1} - u_{m+1}^n)\bigr)} \biggr\}.
\end{aligned}
\end{equation}
Taking the real part of \eqref{RHS_Discrete_NLSV_term_1_EC} yields
\begin{equation}\label{Re_Sum_RHS_Discrete_NLSV_term_1_EC}
\Re{\biggl\{\sum_{m \in \Z} \Tilde{D}_m^n\biggr\}} = \frac{1}{4} \sum_{m \in \Z} \biggl\{ \biggl|\frac{u_{m+1}^{n+1} - u_m^{n+1}}{\Delta x}\biggr|^2 - \biggl|\frac{u_{m+1}^n - u_m^n}{\Delta x}\biggr|^2\biggr\}.
\end{equation}
Further, multiplying the second term on the RHS of \eqref{Descrete_NLSV} by $(\bar{u}_m^{n+1} - \bar{u}_m^n)$, we have
\begin{equation}\label{RHS_Discrete_NLSV_term_2_EC}
\begin{aligned}
\Tilde{V}_m^n &= \biggl[{V(m\Delta x)} - \frac{|u_m^{n+1}|^2 + |u_m^n|^2}{2}\biggr]\frac{(u_m^{n+1} + u_m^n)}{2}(\bar{u}_m^{n+1} - \bar{u}_m^n) \\
&= \biggl[\frac{V(m\Delta x)}{2} - \frac{|u_m^{n+1}|^2 + |u_m^n|^2}{4}\biggr]\bigl(|\bar{u}_m^{n+1}|^2 - |\bar{u}_m^n|^2 + 2i\Im({u_m^{n+1}\bar{u}_m^n)}\bigr).
\end{aligned}
\end{equation}
Thus, summing \eqref{RHS_Discrete_NLSV_term_2_EC} over all $m \in \Z$ and taking the real part, we have
\begin{equation}\label{Re_Sum_RHS_Discrete_NLSV_term_2_EC}
\Re{\biggl\{\sum_{m \in \Z} \Tilde{V}_m^n\biggr\}} = \frac{1}{2} \sum_{m \in \Z}V(m\Delta x)(|u_m^{n+1}|^2 - |u_m^n|^2)  - \frac{1}{4} \sum_{m \in \Z} (|u_m^{n+1}|^4 - |u_m^n|^4).
\end{equation}
Again appealing to \eqref{Descrete_NLSV}, we have that $\Tilde{T}_m^n = \Tilde{D}_m^n + \Tilde{V}_m^n$, and similarly \begin{equation}\label{Re_Sum_Discrete_NLSV}
\Re{\biggl\{\sum_{m \in \Z} \Tilde{T}_m^n\biggr\}} = \Re{\biggl\{\sum_{m \in \Z} \Tilde{D}_m^n\biggr\}} + \Re{\biggl\{\sum_{m \in \Z} \Tilde{V}_m^n\biggr\}}.   
\end{equation}
Hence, combining our results from \eqref{Sum_LHS_Discrete_NLSV_term_EC}, \eqref{Re_Sum_RHS_Discrete_NLSV_term_1_EC}, and \eqref{Re_Sum_RHS_Discrete_NLSV_term_2_EC} via \eqref{Re_Sum_Discrete_NLSV}, we have for all $n \geq 0$,
\begin{align*}
0 =\,& \frac{1}{4} \sum_{m \in \Z} \biggl\{\biggl|\frac{u_{m+1}^{n+1} - u_m^{n+1}}{\Delta x}\biggr|^2 - \biggl|\frac{u_{m+1}^n - u_m^n}{\Delta x}\biggr|^2 \biggr\}\\
&+ \frac{1}{2} \sum_{m \in \Z}V(m\Delta x)(|u_m^{n+1}|^2 - |u_m^n|^2)  - \frac{1}{4} \sum_{m \in \Z} (|u_m^{n+1}|^4 - |u_m^n|^4),
\end{align*}
and from this result, we can immediately deduce \eqref{DiscEnergyConservation}. Therefore, our difference scheme satisfies the discrete conservations laws \eqref{DiscMassConservation} and \eqref{DiscEnergyConservation}.
\end{proof}

\section{The implementation scheme}\label{S: computingsoliton-potential}

To compute soliton-potential interactions, we utilize our adaptation of the finite-difference algorithm described in \cite{MR645840} for one spatial dimension. Of course, to facilitate computation we must truncate the vectors $u^n$ to finitely many points in space and may also only compute over a finite number of time-steps. However, since both the solitons and potentials of concern decay rapidly towards $\pm\infty$, we found it both sufficient and practical to impose artificial boundary conditions by defining a wide spatial interval centered at $0$ and setting $u_0 = u_{M+1} = 0$ at every time-step.

Furthermore, although the authors of \cite{MR645840} use implementations which employ a moving frame of reference to track the center of the soliton data in space, we found this unnecessary provided the interaction was modeled on a sufficiently large space interval for a sufficiently short time. However, in certain settings, an implementation with such a moving frame of reference may enhance the preservation of the conservation properties and allow for numerical studies over a longer time-frame, so this may be an interesting avenue of future exploration.

Now, to implement our finite-difference algorithm, we must first rewrite the scheme given in \eqref{Descrete_NLSV} in a matrix form. We thus define the matrices $I, A,$ and $V$ each with dimension $(M + 1) \times (M + 1)$, where $M + 1$ is the number of space steps over the finite spatial interval defined previously. Thus, we let $I$ be the identity matrix (where we then set the ${1,1}$ and ${M + 1,M + 1}$ entries to $0$). Further, we let $A$ be the matrix associated with the discrete Laplacian operator. Since we are in the case of one spatial dimension, $A$ is tri-diagonal, as noted in \cite{MR645840}; specifically, for $A = [a_{j,k}] \in \C^{(M + 1)\times(M + 1)}$, where 
\[
a_{j,k} = \begin{cases}
    \frac{-2}{|\Delta x|^2}, & \text{if} \quad j = k, \qtq{with} j,k \in \{2,\dots, M\} \\
    \frac{1}{|\Delta x|^2}, & \text{if} \quad |j - k| = 1, \qtq{with} j,k \in \{2,\dots, M\} \\
    0, &\text{otherwise}.
\end{cases}
\]
Finally, let $V$ be a diagonal matrix associated with the potential. Explicitly, for matrix $V = [v_{j,k}]$ where $j,k \in \{1, \dots, M\}$, let $v_{m,m} = V(m \Delta x)$ with $m \in \{2,\dots, M\}$, and $v_{j,k} = 0$ otherwise. Note this construction sets the boundary elements in each matrix to $0$ and hence also aids in enforcing the boundary conditions within the linear portion of the scheme.

We next consider the treatment of the nonlinear term in \eqref{Descrete_NLSV}. For $u^n$ known, we define a mapping $F_n(u)$ from $\C$ to $\C$ with the $m$-th component given by

\begin{equation*}
    F_n(u)_m =\begin{cases}
        -\frac{1}{2}\bigl[{|u_m|^2 + |u_m^n|^2}\bigr](u_m + u_m^n) &\text{if} \quad u_m \neq u_m^n,\\
        -|u_m|^2u_m &\text{if} \quad u_m = u_m^n.
    \end{cases} 
\end{equation*}

Now, we can rewrite our numerical scheme \eqref{Descrete_NLSV} to obtain the form
\begin{equation}\label{MatrixForm}
\frac{i}{\Delta t}\bigl(u_m^{n+1} - u_m^n\bigr) + \biggl(\frac{A}{2} - V\biggr)\biggl(\frac{u_m^{n+1} + u_m^n}{2}\biggr) - F_n(u^{n+1}) = 0.
\end{equation}
To handle the nonlinearity appropriately, we solve the system given in \eqref{MatrixForm} using a successive approximation method, as suggested by \cite{MR645840}. First, to solve $u^{n+1}$ from $u^{n},$ we consider the sequence $\{u_p\}_{p \in \{0,1, \dots\}}$ of complex vectors of size $M + 1$ with $u_0 = u^n$.
Explicitly, to compute $u^1$, compute $u^0$ using the initial data given in \eqref{NLS} and the consequential recursive relation
\begin{equation*}
    u_{p+1} = \biggl[\frac{i}{\Delta t}I \, +\,\frac{A}{4} \,-\, \frac{V}{2}\biggr]^{-1}\biggl[\biggl(\frac{i}{\Delta t}I \, -\,\frac{A}{4} \,+\, \frac{V}{2}\biggr)u^n + F_n(u_p)\biggr].
\end{equation*}

For the stopping point on this recursion, we checked if the mass and energy of $u_{p+1}$, calculated using \eqref{DiscMassConservation} and \eqref{DiscEnergyConservation}, respectively, were conserved both to the desired degree. If the iteration meets these conditions, we set $u^{n+1} = u_{p+1}$. 

\section{Examples of soliton-potential interaction models across several velocity regimes}\label{S: Examples}

We now present some outputs that we have achieved using this method with the conditions specified in each example below. As previously stated, for each example we used a geometric potential modeled from \cite{PhysRevLett.129.040401} with parameters $h = 0.3$, $m = 0.5$, and $w = 0.1$; however, other potentials may be explored with the algorithm. We imposed a mass tolerance of $1 \times 10^{-6}$, and we enforced energy conservation to $1 \times 10^{-4}$. Furthermore, we found that mesh sizes $\Delta t = 0.001$ and $\Delta x = 0.01$ were sufficient to achieve apparent convergence for the following examples. Finally, in each of the scenarios below, we set $x_0 = -5$.

\subsection{The high-velocity regime}

As an example of the high-velocity regime, we set $v = 10, -15 \leq x \leq 15, \text{ and } 0 \leq t \leq 1.2$. 
\begin{figure}[h]
\centering
\makebox[\textwidth][c]{\includegraphics[width = 1\linewidth]{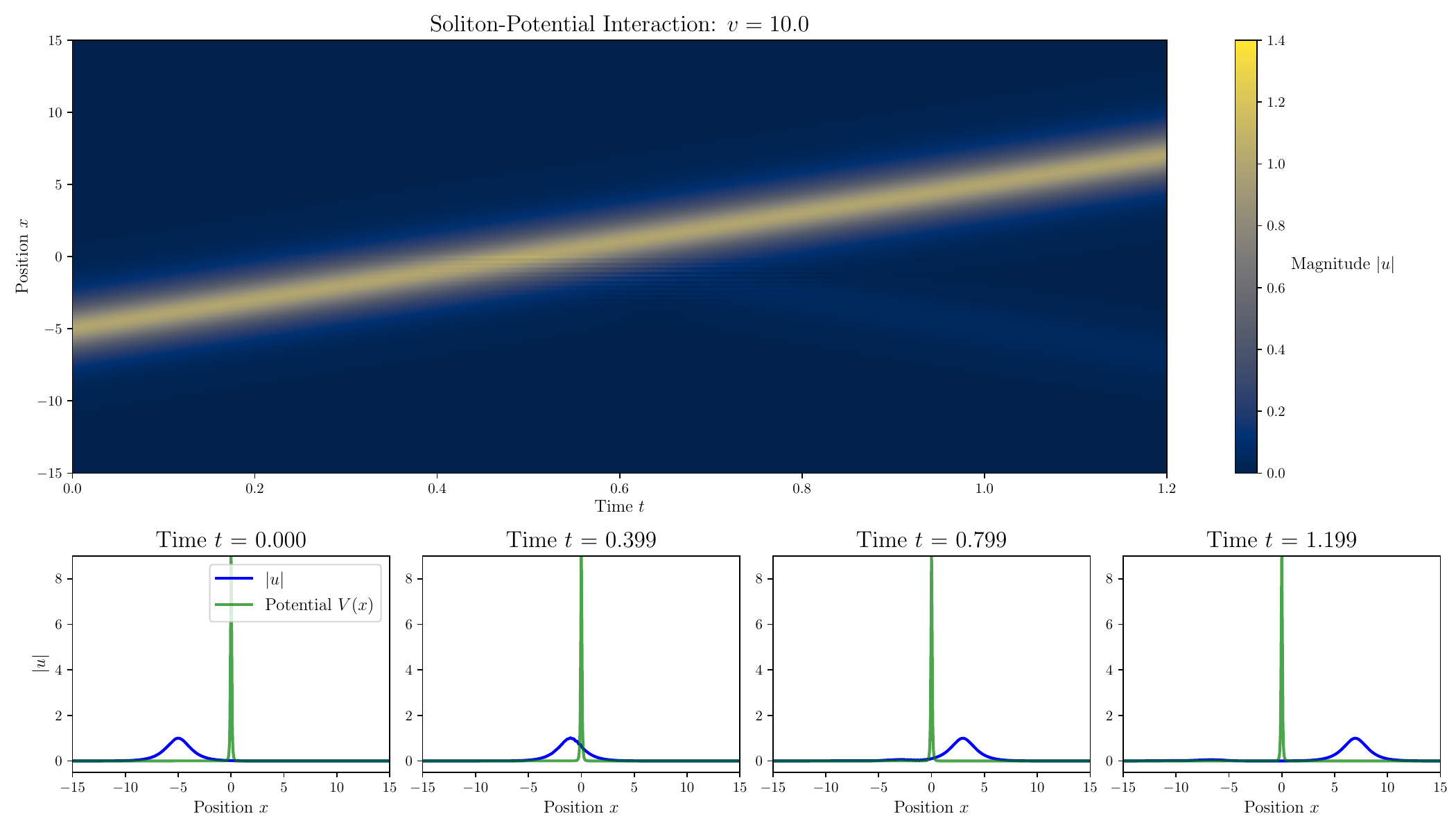}}
\caption{The High-Velocity Case ($v = 10$)}\label{SolPotent10p00}
\end{figure}
\textsc{Figure}~\ref{SolPotent10p00} shows the color-map and four equally spaced frames of the soliton-potential interaction with these parameters. From this example, we see that for $v = 10$, the soliton is essentially transmitted through the potential. This behavior corresponds to the transmission of high velocity solitons with a rapidly decaying potential, as shown in \cite{MR4714537}.

We note that for fast soliton cases both mass energy may be enforced to a higher degree.
\begin{figure}[h]
\centering
\begin{subfigure}[b]{0.49\textwidth}
\makebox[\textwidth][c]{\includegraphics[width = 1\linewidth]{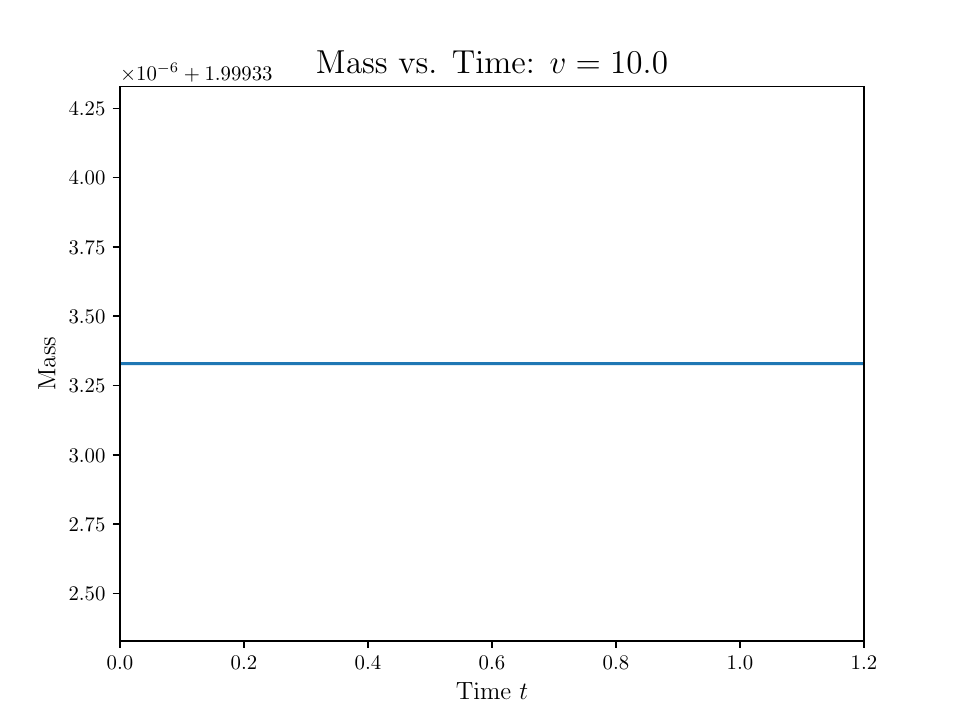}}
\caption{Mass Conservation ($v = 10$)}
\end{subfigure}
\begin{subfigure}[b]{0.49\textwidth}
\makebox[\textwidth][c]{\includegraphics[width = 1\linewidth]{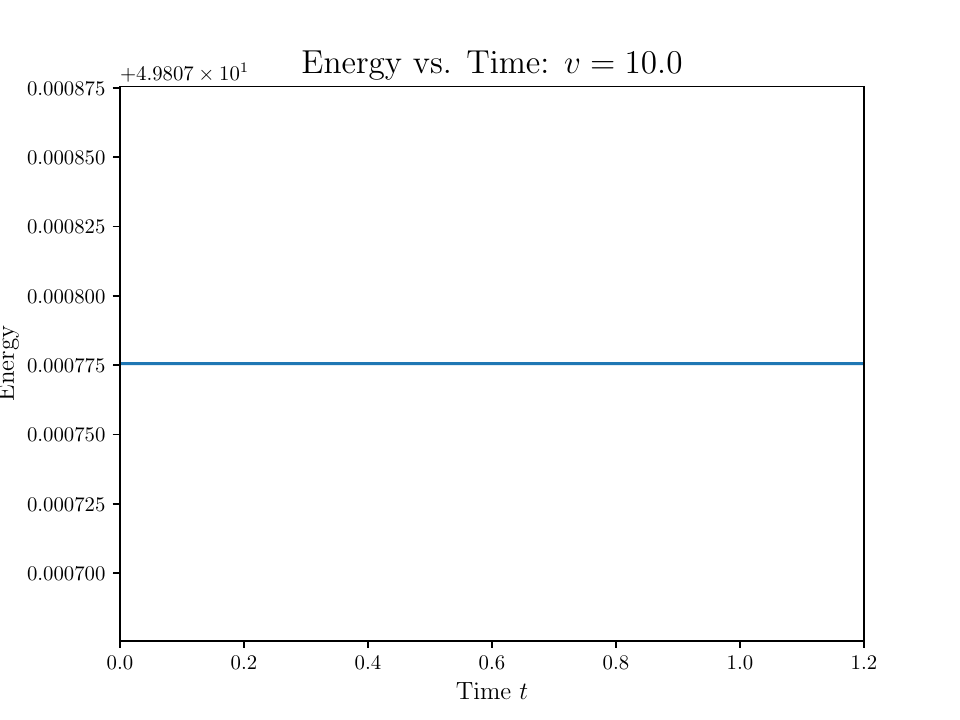}}
\caption{Energy Conservation ($v = 10$)}
\end{subfigure}
\caption{Tracking Mass and Energy Conservation ($v = 10$)}\label{MCEC10p00}
\end{figure}
\textsc{Figure}~\ref{MCEC10p00} shows the evolution of our mass energy calculations for $v = 10$. We see that, at this resolution, there is little to no change in these invariants.

\subsection{The low-velocity regime}

As an example of the low-velocity regime, we take $v = 0.96, -15 \leq x \leq 15, \text{ and } 0 \leq t \leq 12.5$.
\begin{figure}[h]
\centering
\makebox[\textwidth][c]{\includegraphics[width = 1\linewidth]{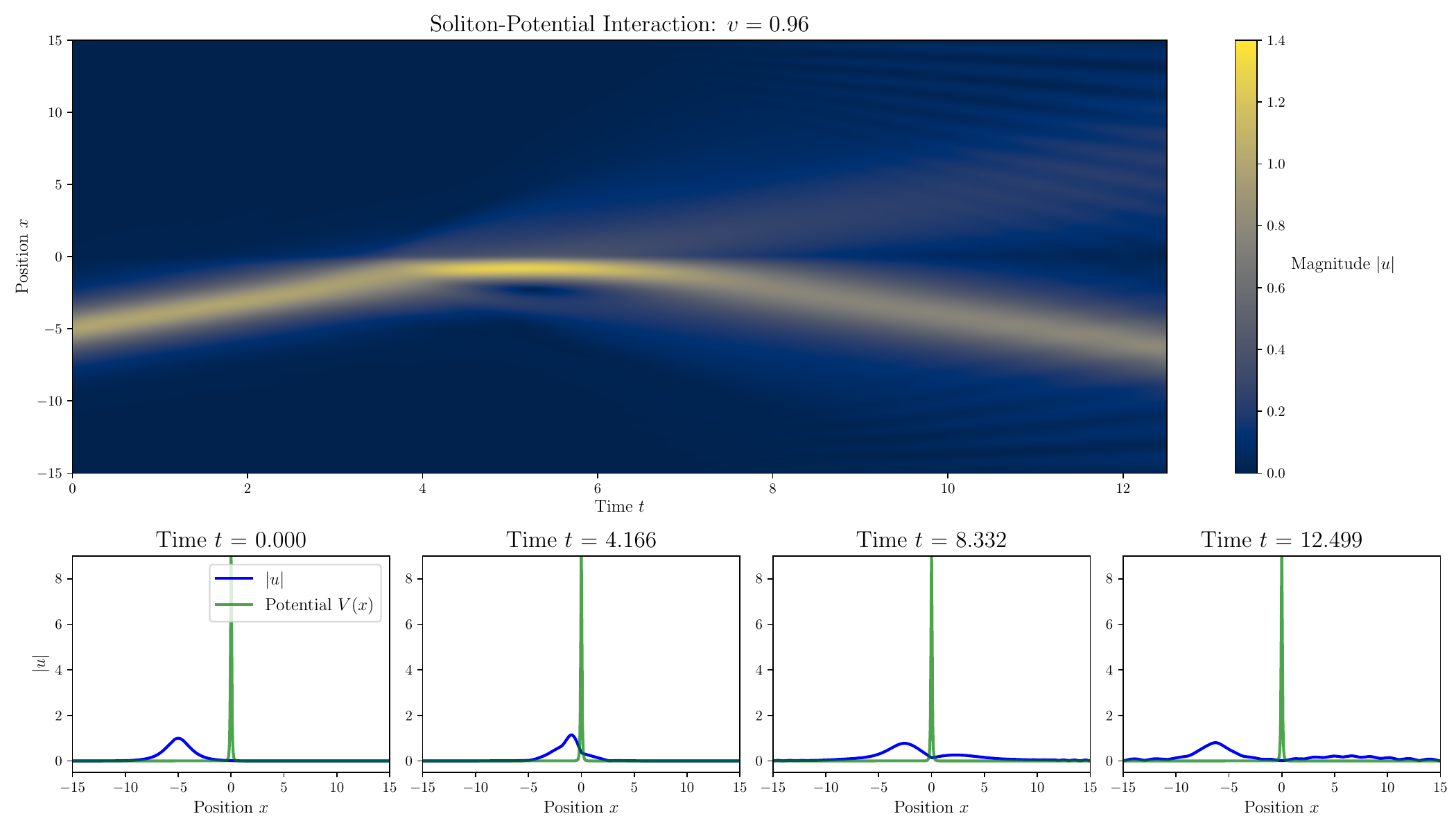}}
\caption{The Low-Velocity Case ($v = 0.96$)}\label{SolPotent0p96}
\end{figure}
In \textsc{Figure}~\ref{SolPotent0p96}, we illustrate the soliton-potential interaction under these parameters via a color-map and four equally spaced frames, as before. For $v = 0.96$, we see via \textsc{Figure}~\ref{SolPotent0p96} that the soliton is mostly reflected. Thus, this serves as our example for the low-velocity case, which was addressed at the theoretical level in the recent result \cite{caliaro2026confinementorbitalstabilitysolitons}.

For low-velocity cases, we observe greater mass and energy fluctuation, as shown in \textsc{Figure}~\ref{MCEC0p96}. 
\begin{figure}[h]
\centering
\begin{subfigure}[b]{0.47\textwidth}
\makebox[\textwidth][c]{\includegraphics[width = 1\linewidth]{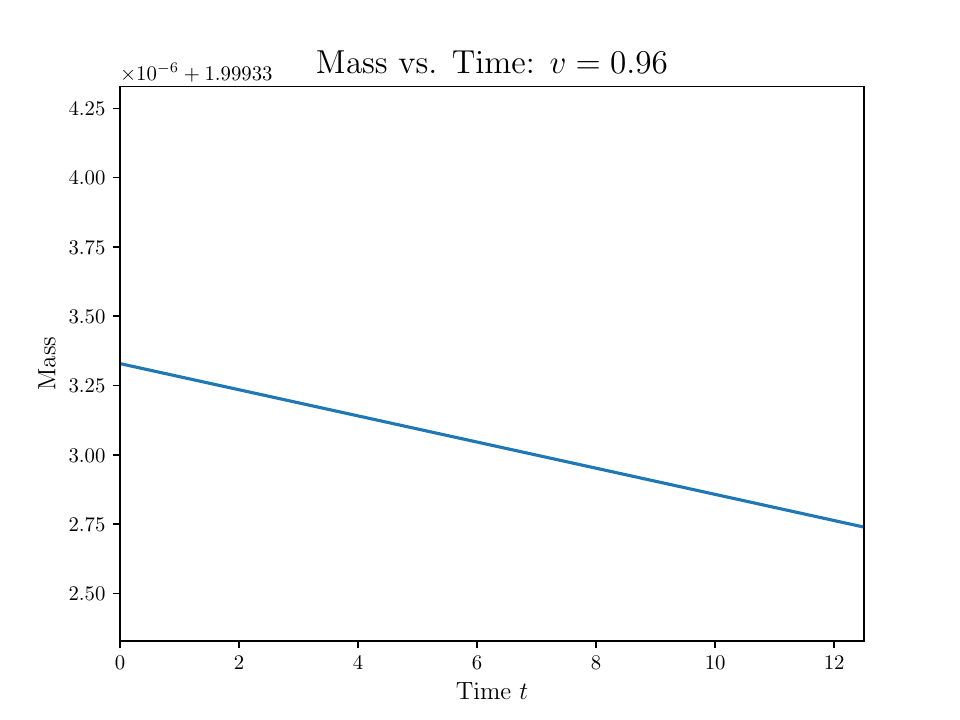}}
\caption{Mass Conservation ($v = 0.96$)}
\end{subfigure}
\begin{subfigure}[b]{0.47\textwidth}
\makebox[\textwidth][c]{\includegraphics[width = 1\linewidth]{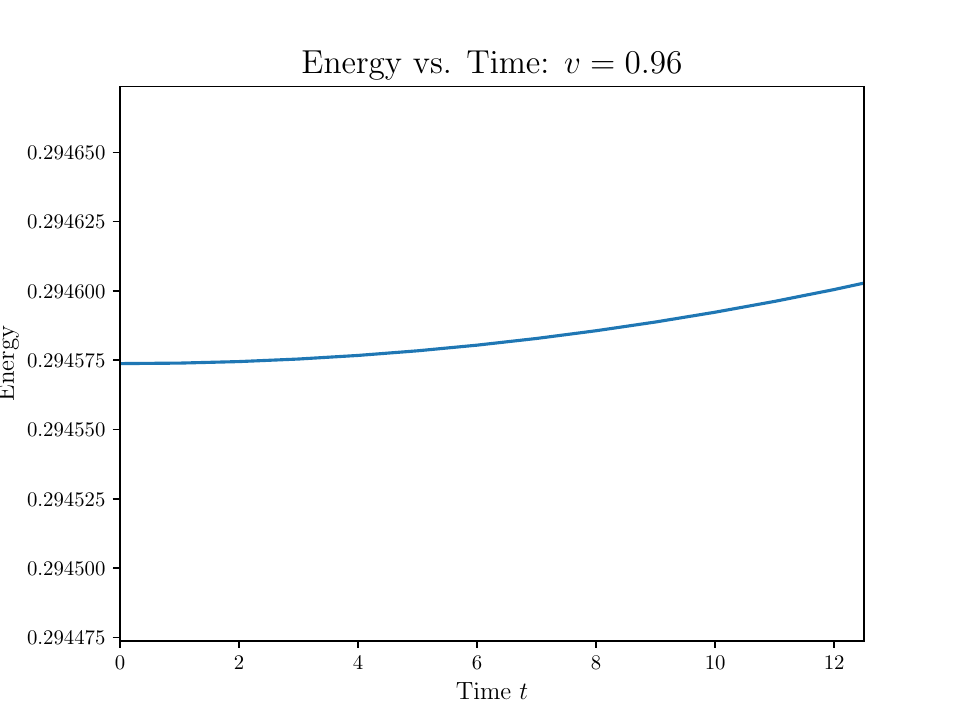}}
\caption{Energy Conservation ($v = 0.96$)}
\end{subfigure}
\caption{Tracking Mass and Energy Conservation ($v = 0.96$)}\label{MCEC0p96}
\end{figure}
However, any changes in these theoretical invariants are well within the desired bounds. It is interesting to note that the mass seems to exhibit a linear decrease and the energy seems to exhibit a superlinear increase over the time interval. We expect this accumulation of errors in the invariants arises as a result of the artificial boundary conditions imposed at the ends of the spatial interval; these small boundary errors seem to propagate as ripples from the edges of the spacial domain (as seen in the later times in \textsc{Figure}~\ref{SolPotent0p96}) and are likely more prominent in this case due to the time interval required to model soliton-potential interaction.

\subsection{The intermediate-velocity regime}

\begin{figure}[h]
\centering
\makebox[\textwidth][c]{\includegraphics[width = 1\linewidth]{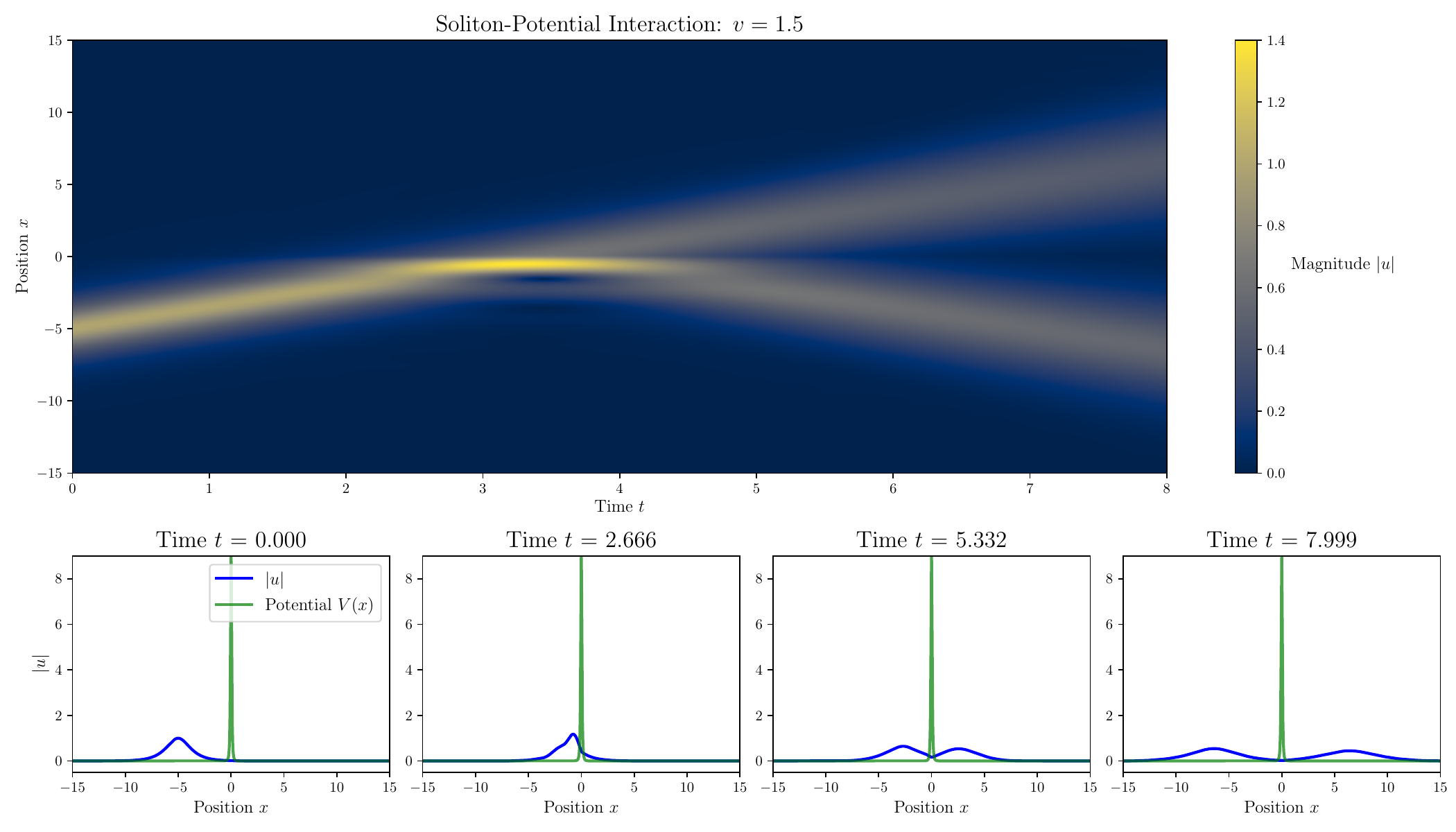}}
\caption{The Intermediate-Velocity Case ($v = 1.5$)}\label{SolPotent1p50}
\end{figure}
Finally, we examine the intermediate-velocity regime. To this end, we let $v = 1.5, -15 \leq x \leq 15, \text{ and } 0 \leq t \leq 8$. Once again, we provide the color-map and four equally spaced frames of the soliton-potential interaction under these parameters in \textsc{Figure}~\ref{SolPotent1p50}, where we observe an almost equal amount of transmission and reflection. Thus, this example exhibits splitting, as expected in the intermediate-velocity case.

Similar to the previous example, we observe that the mass and energy are relatively well-conserved at this resolution, as shown in \textsc{Figure}~\ref{MCEC1p50}. 
\begin{figure}[h]
\centering
\begin{subfigure}[b]{0.47\textwidth}
\makebox[\textwidth][c]{\includegraphics[width = 1\linewidth]{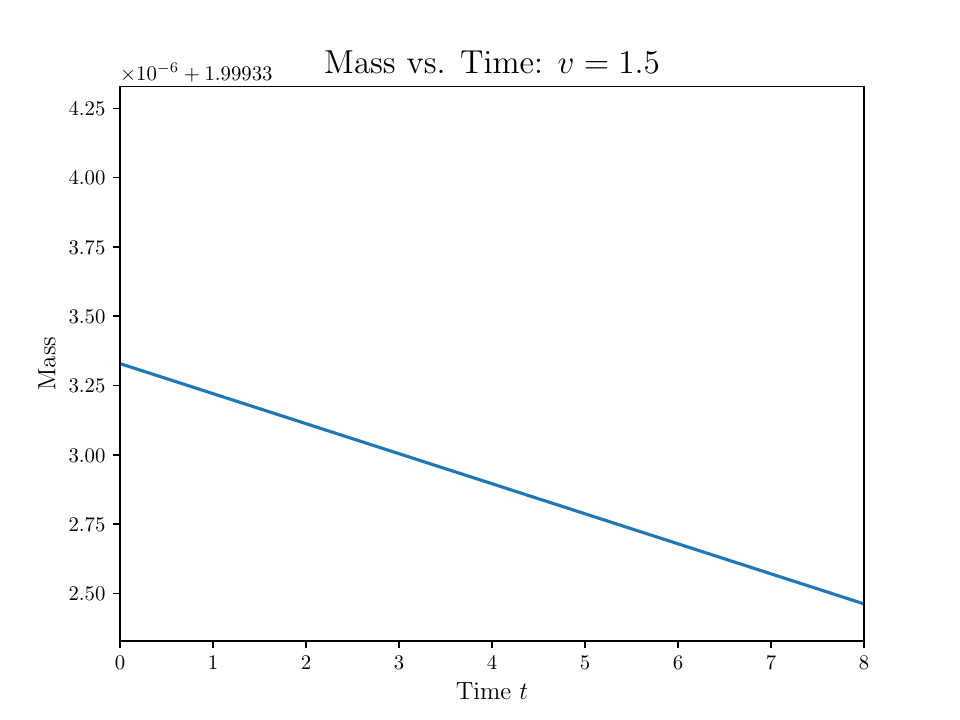}}
\caption{Mass Conservation ($v = 1.5$)}
\end{subfigure}
\begin{subfigure}[b]{0.47\textwidth}
\makebox[\textwidth][c]{\includegraphics[width = 1\linewidth]{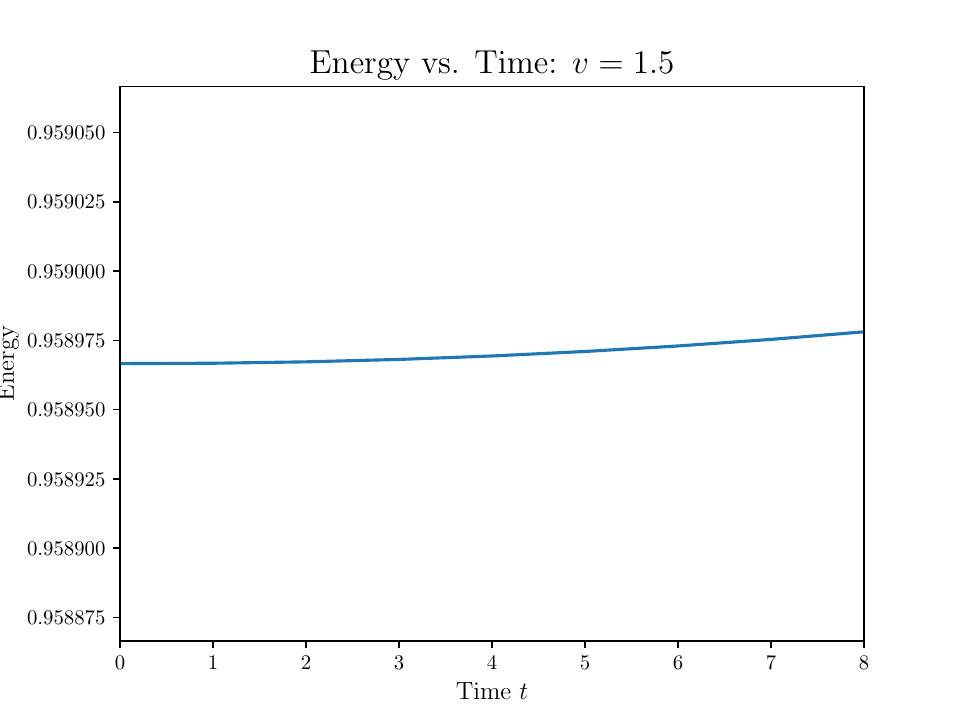}}
\caption{Energy Conservation ($v = 1.5$)}
\end{subfigure}
\caption{Tracking Mass and Energy Conservation ($v = 1.5$)}\label{MCEC1p50}
\end{figure}
We again observe that the mass and energy vary in similar fashion to the low-velocity case.

\section{Conclusion}

We adapted a finite-difference method (from \cite{MR645840}) to accommodate our NLS model with a more general spatially dependent potential, and showed that this adaptation satisfies the appropriate discrete mass and energy conservation laws. Additionally, we implemented this finite difference method in a manner which intrinsically relies on preservation of the mass and energy invariants. We then applied the algorithm to soliton data across high-, low-, and intermediate-velocity regimes. Finally, we examined the computed results in the context of the literature and observed that, for sufficiently small $\Delta x$ and $\Delta t$, the numerical results strongly matched the expected soliton-potential interaction behavior across these three regimes and also exhibited clear preservation of the mass and energy invariants to an acceptable degree.
\bibliography{Working_References}

\end{document}